\documentclass{raex}
\usepackage{amsmath, amssymb, amsthm, amsfonts}

\usepackage[utf8]{inputenc}
\usepackage[T1]{fontenc}
\usepackage[english]{babel}
\usepackage{url}
\usepackage{mathtools}

\makeatletter
\newcommand{\ThanksMark}{\sbox{\thanks@box}{\footnotemark[1]}}
\makeatother

\begin{Author}
	\FirstName{Mykola}\LastName{Pratsiovytyi}
	\PostalAddress{Institute of Mathematics of NAS of Ukraine, 3 Tereshchenkivska St., Kyiv 01024, Ukraine; Dragomanov Ukrainian State University, 9 Pyrohova St., Kyiv 01601, Ukraine}
	\Email{prats4444@gmail.com}
	\Thanks{This work was supported by a grant from the Simons Foundation (SFI-PD-Ukraine-00014586, M.P., O.V.).}
\end{Author}

\begin{Author}
	\FirstName{Yaroslav}\LastName{Vynnyshyn}
	\PostalAddress{Kyiv Natural Science Lyceum No.~145, 46 Shota Rustaveli St., Kyiv 01033, Ukraine}
	\Email{yartan@gmail.com}
\end{Author}

\begin{Author}
	\FirstName{Oleh}\LastName{Vynnyshyn}
	\PostalAddress{Institute of Mathematics of NAS of Ukraine, 3 Tereshchenkivska St., Kyiv 01024, Ukraine}
	\Email{oleh.vynnyshyn@imath.kiev.ua}
    \ThanksMark
\end{Author}

\begin{MathReviews}
	\primary{40A05}
	\secondary{11B05, 11A63, 28A80}
\end{MathReviews}

\begin{KeyWords}
\end{KeyWords}

\title{A series with a Cantorval set of subsums and ratio of consecutive terms tending to $1/2$ (the second Jones problem)}

\theoremstyle{plain}
\newtheorem{theorem}{Theorem}[section]
\newtheorem{lemma}[theorem]{Lemma}
\newtheorem{corollary}[theorem]{Corollary}
\theoremstyle{definition}
\newtheorem{definition}[theorem]{Definition}
\newtheorem{example}[theorem]{Example}
\theoremstyle{remark}
\newtheorem{remark}[theorem]{Remark}

\begin{document}

\maketitle

\begin{abstract}
The second Jones problem (2011) consists in determining whether there exists a convergent positive series for which the ratio of consecutive terms has a limit and the set of subsums is a Cantorval; in the survey by G{\l}{\k a}b and Prus-Wi\'sniowski (2026) this problem is included in the list of open problems. In this paper, a family of such series is constructed explicitly as combinations of blocks of the form $\frac{2^k-1}{2^t} + \frac{2^{k-1}}{2^t} +  \frac{2^{k-2}}{2^t} + \cdots  + \frac{2}{2^t}$, where $t$ is the index of the last term of the block in the series and $k\ge2$ is the length of the block.
For each series of the family it is proved that its set of subsums contains an interval and is not a finite union of closed intervals, and hence is a Cantorval. The proof relies on mixed-radix numerical systems with a redundant alphabet and uses an idea of Guthrie and Nymann (1988). If the lengths of the blocks tend to infinity, then the limit of the ratio of consecutive terms of such series is $\tfrac{1}{2}$, which, by a result of Jones, is the only possible value of the limit for which the set of subsums of a series can be a Cantorval. Thus, an affirmative answer to the second Jones problem is obtained.
\end{abstract}

\section{Introduction}
Let $D= \{0,1\}$ be a two-symbol alphabet and let $L = D \times D \times D \times \cdots$ be the space of infinite sequences of elements of the alphabet $D$.

\begin{definition}
The \emph{set of subsums} (\emph{achievement set}) of the series defined by a sequence of real terms $(x_n)$ is the set
\[ E(x_n) \coloneqq
\left\{
x \;\middle|\;
x = \sum_{n=1}^{\infty} c_n x_n, \quad
(c_n) \in L
\right\}. \]
\end{definition}

The set of subsums of a convergent positive series is compact and perfect~\cite{Kakeya1914}. Kakeya also raised the question of whether it is always either homeomorphic to the Cantor set or a finite union of closed intervals. A complete classification was given by Guthrie and Nymann.

\begin{theorem}[Guthrie--Nymann classification theorem~{\cite[Theorem~1]{GuthrieNymann_1},~\cite{NymannSaenz_2}}]
\label{Hatri_Nymann_th}
The set of subsums of a convergent positive series is

\begin{enumerate}
\item 
either a finite union of closed intervals,
\item
or homeomorphic to the Cantor set,
\item 
or a Cantorval, that is, a set homeomorphic to the set of subsums of the Guthrie--Nymann series $\frac{3}{4} + \frac{2}{4} + \frac{3}{4^2} + \frac{2}{4^2} + \cdots + \frac{3}{4^k} + \frac{2}{4^k} + \cdots$, which is the union of a nowhere dense set and infinitely many closed intervals.

\end{enumerate}
\end{theorem}

The term ``Cantorval'' was introduced in~\cite{MendesOliveira1994}. Since then, many series and families of series whose sets of subsums are Cantorvals have been described: multigeometric series~\cite{Bartosz_multigeom}, their generalizations~\cite{KMV2024,karvatskyi_2026, Prats_Karv_Makarchuk} and others~\cite{VMPS2019, Nowakowski_OnANewCondition}. Jones~\cite{Jones2011} drew attention to the role of the ratio of consecutive terms: if for a convergent positive series $\sum x_n$ the limit $l \coloneqq \lim_{n \to \infty} \frac{x_{n+1}}{x_n}$ exists, then, by~\cite[Proposition~3.3]{Jones2011}, for $l < \frac{1}{2}$ the set of subsums of the series is homeomorphic to the Cantor set, and for $\frac{1}{2} < l \leq 1 $ it is a finite union of closed intervals. Hence, a Cantorval is possible only for $l=\frac{1}{2}$.
The question of the existence of
a series for which the limit $l = \frac{1}{2}$ exists and the set of subsums is a
Cantorval is known as the \emph{second Jones problem}
(\cite[p.~519]{Jones2011};~\cite[Problem~4]{Survey2025_Prus-Wisniowski}). 

In the present paper, an affirmative answer to this problem is given; namely, a family of such series is constructed explicitly. In particular, the convergent positive series~\eqref{example_series} belongs to this family, its set of subsums $E(x_n)$ is a Cantorval, and moreover
\[ E(x_n) \supset \Bigl[\frac{3}{4}, 1\Bigr]
\qquad\text{and}\qquad
l = \frac12. \] 

After this paper had been written, we learned about the work~\cite{Metriq_Prism_by_AI} by an artificial intelligence laboratory, posted on Zenodo together with solutions to some other problems from the survey~\cite{Survey2025_Prus-Wisniowski}. Their example, the series $x_{2n-1}=\frac{8n+17}{4^n}$, $x_{2n}=\frac{4n+18}{4^n}$, uses different approaches. The authors declare their result only as a candidate solution, since it was obtained with substantial use of artificial intelligence and has not undergone independent peer review (this is exactly why this work is not on arXiv and why it initially escaped our notice). We have studied this work and consider their result elegant, with a complete proof without gaps. The construction of the present paper is different and was obtained independently, it yields a family of series of cardinality continuum with an explicitly specified interval in the set of subsums for each of the series. 

\section{Construction of the series}
Let $s \geq 2$ be a fixed natural number and let $(k_1, k_2, \ldots, k_m)$ be a fixed tuple of positive integers.
For the multigeometric series~\cite{Bartosz_multigeom}
\[ \frac{k_1}{s} + \cdots + \frac{k_m}{s} + \frac{k_1}{s^2} + \cdots+ \frac{k_m}{s^2} + \cdots\]
we shall use the shorthand notation $(k_1, k_2, \ldots, k_m; \frac{1}{s})$.

For integers $k \ge 2$ we put
\[B_k \coloneqq \{2^{k}-1, 2^{k-1}, 2^{k-2},\ldots,2^2, 2\}.\]

Fix a sequence of integers $\bar{a} = (a_j)$ with $a_j \geq 2$, and put
\[
s_j \coloneqq 2^{a_j}, \qquad
t_0 \coloneqq 0, \qquad 
t_j \coloneqq a_1 + a_2 + \cdots + a_j.
\]

\begin{definition}
By the \emph{block of rank $j$} we mean the expression
\[
\frac{B_{a_j}}{2^{t_j}} 
= \frac{2^{a_j}-1}{2^{t_j}} +
\frac{2^{a_j-1}}{2^{t_j}} +
\frac{2^{a_j-2}}{2^{t_j}} + \cdots +
\frac{2^2}{2^{t_j}} +
\frac{2}{2^{t_j}}. \]
\end{definition}

\begin{definition}
By the \emph{hybrid series with decomposition $\bar{a}$} we mean the series whose terms
are obtained by writing out consecutively the elements of the
\emph{blocks of rank $j$}, $j = 1, 2, \ldots$, that is, $\sum_{j=1}^\infty \frac{B_{a_j}}{2^{t_j}}$. 
\end{definition}

We denote the terms of the hybrid series by $x_1,x_2, \ldots$; the terms of the series belonging to the block of rank $j$ have indices $t_{j-1} + 1, \ldots, t_j$. The hybrid series is convergent, since $x_n \leq \frac{1}{2^{n-1}}$ and $\sum_{n=1}^{\infty} x_n < 2$.

\begin{remark}
For $a_j=2$ for all $j$ we obtain the Guthrie--Nymann series $(3,2;\tfrac{1}{4})$, whose set of subsums is known to be a Cantorval.
\end{remark}

\begin{remark}
If $a_j = k$ for all $j$, then the hybrid series is exactly the
multigeometric series $(2^{k}-1, 2^{k-1}, \ldots, 4, 2; \frac{1}{2^{k}})$. It has already been proved that the set of subsums of such a series is a Cantorval~\cite{karvatskyi_2026}.
\end{remark}

\begin{example}
For $\bar{a} = (2,3,4,5,\ldots)$, that is, $a_j = j+1$, the terms of the hybrid series have the form
\begin{equation}
\label{example_series}
\frac{B_2}{2^2}, \frac{B_3}{2^{2+3}}, \ldots, \frac{B_{n+1}}{2^{t_n}}, \ldots ,
\end{equation}
namely:
\[
\frac{3}{4},\ \frac{1}{2},\ \Big| \
\frac{7}{32},\ \frac{1}{8},\ \frac{1}{16},\ \Big| \
\frac{15}{512},\ \frac{1}{64},\ \frac{1}{128},\ \frac{1}{256},\ \Big| \
\frac{31}{16384},\ \frac{1}{1024},\ 
\ldots
\]
\end{example}

\begin{lemma}\label{ratio}
For the hybrid series with decomposition $\bar{a}$ such that $a_j \to \infty$ as $j \to \infty$, we have $\displaystyle \lim_{n\to\infty} \frac{x_{n+1}}{x_n} = \frac12$.
\end{lemma}

\begin{proof}
Let $x_n$ belong to the block of rank $j$. If $x_n$ is the first element of the block of rank $j$, then
\[\frac{x_{n+1}}{x_n}  = 
\frac{2^{a_j-1}}{2^{a_j}-1} = 
\frac{1}{2} \cdot \frac{2^{a_j}-1+1}{2^{a_j}-1} =
\frac{1}{2} + \frac{1}{2(2^{a_j}-1)}.\]
If $x_n$ is the last element of the block of rank $j$, then
\[\frac{x_{n+1}}{x_n} = 
\frac{2^{a_{j+1}}-1} {2^{t_{j+1}}}
\cdot \frac{2^{t_j}}{2} =
\frac{2^{a_{j+1}}-1} {2^{a_{j+1}}} \cdot \frac{1}{2} =
\frac{1}{2} - \frac{1}{2^{a_{j+1}+1}} .\]
In all other cases $ \frac{x_{n+1}}{x_n} = \frac{1}{2}$.
Hence, for all $n$,
\[
\Big| \frac{x_{n+1}}{x_n} - \frac{1}{2} \Big| \leq 
 \frac{1}{2(2^{a_j}-1)}+\frac{1}{2^{a_{j+1}+1}} \to 0.
\]
\end{proof}

\begin{definition}
By the \emph{alphabet of rank $m$} we mean the set of subsums of the block $B_m$:
\begin{multline*}
A_m \coloneqq \Bigl\{ \sum_{b \in C} b \;\Big|\; C \in 2^{B_m} \Bigr\} = \\ 
= \Bigl\{\alpha \; : \; \alpha = c_1(2^m - 1) + c_2 2^{m-1} + \ldots 
+c_{m-1} 2^2 + c_m 2, \quad c_j \in D  \Bigr\}.
\end{multline*}
\end{definition}

For rank $m$ we have the base $s =2^m$, and the alphabet $A_m$ consists of integers: all even numbers of the interval $[0, s-2]$ (they are obtained as subsums of the even numbers $\{2,4,\ldots, 2^{m-1}\}$) and all
odd numbers of the interval $[s-1, 2s-3]$ (if the element $2^m -1$ enters the sum). The alphabet $A_m$ contains exactly one pair of consecutive integers, namely $s-2$ and $s-1$.

\section{Existence of an interval in the set of subsums}

We prove that the set of subsums of the hybrid series with decomposition $\bar{a}$ (in particular, of the series~\eqref{example_series}) contains an interval.
The proof reproduces the scheme of
Guthrie and Nymann~\cite[Section~2]{GuthrieNymann_1}, adapted to the number representation system with the sequence of bases $(s_1, s_2, \ldots)$, $s_j = 2^{a_j}$, and the alphabets $(A_{a_1}, A_{a_2}, \ldots)$. In this system the weight of the $j$-th
digit position equals $\frac{1}{2^{t_j}}$.

\begin{definition}
By a \emph{finite fraction} of length $k$ we mean a number
\[
x = 
\sum_{j=1}^{k} \frac{\delta_j}{2^{t_j}}
=: (\delta_1, \delta_2, \ldots, \delta_k) ,
\qquad \delta_j \in \{0, 1, \ldots, s_j - 1\},
\]
that is, $(\delta_1, \delta_2, \ldots, \delta_k)$ is the finite representation of the number $x$ in the number representation system with bases
$(s_1, s_2, \ldots)$, $s_j = 2^{a_j}$. 
\end{definition}

For $c_j \in A_{a_j}, c_j \leq s_j -1$, the finite fraction $( c_1, \ldots, c_k)$ is a subsum of the hybrid series $\sum x_n$, since each $c_j$ corresponds directly to a certain set of summands from the block ${B_{a_j}}$.

\begin{lemma}\label{points_in_interval_in_E}
Every finite fraction $(\delta_1, \ldots, \delta_k)$ with
$\delta_1 = s_1 - 1$ is a subsum of the sum of the first $k$ blocks $(\frac{B_{a_1}}{2^{t_1}}, \ldots, \frac{B_{a_k}}{2^{t_k}})$ of the hybrid series $\sum x_n$.
\end{lemma}

\begin{proof}
Induction on $k$.

\emph{Base case} ($k = 1$): $(s_1 - 1) \in A_{a_1}$.

\emph{Induction step}. Suppose the statement holds for all lengths $1, \ldots, k$, and
let $x = (\delta_1, \ldots, \delta_{k+1})$, $\delta_1 = s_1 - 1$.

If $\delta_{k+1} \in A_{a_{k+1}}$, that is,
$\delta_{k+1}$ is even or $\delta_{k+1} = s_{k+1} - 1$, then the statement holds by the induction hypothesis. 
Let $\delta_{k+1} \notin A_{a_{k+1}}$; then $\delta_{k+1}$ is an odd number from the interval $[1, s_{k+1} - 3 ]$.
Let $1 \le m \le k$ be the largest index for which
$\delta_m > 0$ (such an $m$ exists, since $\delta_1 > 0$); then
$\delta_{m+1} = \dots = \delta_k = 0$. 
Since
\[
\frac{1}{2^{t_m}}
= \sum_{i=m+1}^{k} \frac{s_i - 1}{2^{t_i}} + \frac{1}{2^{t_{k}}}
= \sum_{i=m+1}^{k} \frac{s_i - 1}{2^{t_i}} + \frac{s_{k+1}}{2^{t_{k+1}}},
\]
we have
\begin{multline*}
(\delta_1, \ldots, \delta_{m}, \underbrace{0, \ldots, 0}_{k-m}, \delta_{k+1}) = \\
(\delta_1, \ldots, \delta_{m-1},\ \delta_m - 1)
+ \sum_{i=m+1}^{k} \frac{s_i - 1}{2^{t_i}}
+ \frac{s_{k+1} + \delta_{k+1}}{2^{t_{k+1}}} =  \\
(\delta_1, \ldots, \delta_{m-1},\ \delta_m - 1, s_{m+1} -1, \ldots, s_k -1) + \frac{s_{k+1} + \delta_{k+1}}{2^{t_{k+1}}} .
\end{multline*}
Here $s_i - 1 \in A_{a_i}$, 
and $s_{k+1} + \delta_{k+1}$ is an odd number from the interval $[s_{k+1}+1,\ 2s_{k+1}-3]$; hence
$s_{k+1} + \delta_{k+1} \in A_{a_{k+1}}$.

For $m \ge 2$, the fraction $(\delta_1, \ldots, \delta_{m-1},\ \delta_m - 1)$
is a subsum of the sum of the first $k$ blocks $(\frac{B_{a_1}}{2^{t_1}}, \ldots, \frac{B_{a_k}}{2^{t_k}})$ by the induction hypothesis. If $m = 1$, then $\delta_1 - 1 = s_1 - 2 \in  A_{a_1}$.
\end{proof}

\begin{theorem}\label{e_contains_interval}
For the hybrid series with decomposition $\bar{a}$ we have 
\[E(x_n) \supset \Bigl[1 - \frac{1}{2^{a_1}}, 1\Bigr].\]
\end{theorem}

\begin{proof}
Every number $y \in [1 - \frac{1}{2^{a_1}}, 1]$ has a representation of the form $y = \sum_{j =1}^{\infty} \frac{\delta_j}{2^{t_j}}$, where $\delta_1 = s_1 -1$, $\delta_j \in \{0,\ldots, s_j - 1 \}$.
 The sequence of partial finite fractions $(\delta_1, \ldots, \delta_k) \to y$ as $k \to \infty$. By Lemma~\ref{points_in_interval_in_E}, this sequence of finite fractions is contained in $E(x_n)$. Since $E(x_n)$ is a perfect set, and hence closed, we get $[1 - \frac{1}{2^{a_1}}, 1] \subset E(x_n)$.
\end{proof}

\begin{remark}
The interval of Theorem~\ref{e_contains_interval} is exactly the set of numbers
$x \in [0,1]$ whose representation in the number representation system with bases
$(s_1, s_2, \ldots)$ begins with the digit $s_1 - 1$.
\end{remark}

\begin{remark}
For $a_j = 2$ for all $j$, Lemma~\ref{points_in_interval_in_E} and
Theorem~\ref{e_contains_interval} turn into the argument of Guthrie and Nymann~\cite[Section~2]{GuthrieNymann_1} for the series $(3,2;\frac{1}{4})$.
\end{remark}

\section{Cantorvality}

\begin{lemma}\label{lemma_not_intervals}
The set of subsums of the hybrid series is not a finite union of closed intervals.
\end{lemma}

\begin{proof}
If $x_n$ and $x_{n+1}$ belong to the same block, then within the block the denominators of the terms of the series are equal while the numerators decrease, so $x_n > x_{n+1}$. Let $x_n$ be the last element of the block of rank $j$ and $x_{n+1}$ be the first element of the block of rank $j+1$. At the junction of the blocks we have:
\[x_{n+1} = \frac{2^{a_{j+1}} - 1}{2^{t_{j+1}}} = 
\frac{1}{2^{t_j}} \cdot \frac{2^{a_{j+1}} - 1} {2^{a_{j+1}}} < \frac{1}{2^{t_j}} < \frac{2}{2^{t_j}} = x_n.\]
Hence, the sequence of terms of the series is decreasing.

Let $x_k$ be the $m$-th element of the block of rank $j$, $m \geq 2$, that is, $k = t_{j-1} + m$. Then $x_k = \frac{2^{a_j -m +1}}{2^{t_j}} = \frac{1}{2^{t_{j-1} + m - 1}} = \frac{1}{2^{k-1}}$. If $m=1$, then $x_k <\frac{1}{2^{k-1}}$.
Denote by $r_n = \sum_{l=n+1}^{\infty} x_l$ the remainder of the series. 
If $x_n$ is not the first element of a block, then $r_n < \sum_{l=n}^\infty \frac{1}{2^l} = \frac{1}{2^{n-1}} = x_n$.
Hence, the inequality $x_n > r_n$ holds for infinitely many $n$.

By Kakeya's criterion~\cite{Kakeya1914}, for a series with a decreasing sequence of terms for which $x_n > r_n$ infinitely often, the set of subsums $E(x_n)$ is not a finite union of closed intervals.
\end{proof}

\begin{theorem}\label{Cantorvality}
For the hybrid series with decomposition $\bar{a}$, the set of subsums is a Cantorval.
\end{theorem}
\begin{proof}
The existence of an interval (Theorem~\ref{e_contains_interval}) in the set of subsums of the hybrid series means that $E(x_n)$ is not homeomorphic to the Cantor set, while Lemma~\ref{lemma_not_intervals} shows that $E(x_n)$ is not a finite union of closed intervals. Therefore, by the Guthrie--Nymann classification theorem (Theorem~\ref{Hatri_Nymann_th}), $E(x_n)$ is a Cantorval.    
\end{proof}

\begin{corollary}
For the hybrid series with decomposition $\bar{a} = (a_j)$ such that $a_j \to \infty$ as $j \to \infty$, the limit $\lim_{n\to\infty} x_{n+1}/x_n$ exists and equals $\frac{1}{2}$ (Lemma~\ref{ratio}), and the set $E(x_n)$ is a Cantorval (Theorem~\ref{Cantorvality}). In particular, this holds for the series~\eqref{example_series}.
\end{corollary}

This corollary gives an affirmative answer to the second Jones problem (\cite[p.~519]{Jones2011}; \cite[Problem~4]{Survey2025_Prus-Wisniowski}).

\section*{Acknowledgements}
The English translation from Ukrainian original text was done with usage of an AI language model Claude Fable 5.1. We have checked and edited the translation and take full responsibility for the content.

\end{document}